\documentclass{amsart}
\pdfoutput=1
\usepackage{graphicx}

\def\Q{\mathbb{Q}}\def\Z{\mathbb{Z}}\def\R{\mathbb{R}}\def\C{\mathbb{C}}
\def\cal{\mathcal}
\def\ds{\displaystyle}
\def\Res{\operatorname{Res}}
\def\Li{\operatorname{Li}}
\def\E{\operatorname{E}}
\def\VPH{{\vphantom{X^{X^X}}}} 
\def\WPH{{\vphantom{X_{X_X}}}} 

\newcommand\e{\mathrm{e}} 
\newcommand\im{\mathrm{i}} 
\newcommand\dd{\mathrm{d}} 
\newcommand\Oh{\mathrm{O}} 
\newcommand\oh{\mathrm{o}} 

\newtheorem{theorem}{Theorem}
\newtheorem{lemma}{Lemma}
\newtheorem{corollary}{Corollary}
\newtheorem{definition}{Definition}
\newtheorem{proposition}{Proposition}

\begin{document}

\title{Lindel\"of Representations and (Non-)Holonomic Sequences}

\author{Philippe Flajolet}
\address{INRIA Paris-Rocquencourt}
\email{philippe.flajolet at inria.fr}

\author{Stefan Gerhold}
\address{TU Vienna and Microsoft Research-INRIA, Orsay}
\email{sgerhold at fam.tuwien.ac.at}
\thanks{This work was partially supported by CDG, BA-CA, AFFA, and the joint INRIA-Microsoft Research Laboratory.}

\author{Bruno Salvy}
\address{INRIA Paris-Rocquencourt}
\email{bruno.salvy at inria.fr}



\date{June 2, 2009}

\begin{abstract}
Various sequences that possess explicit analytic expressions can be analysed 
asymptotically through integral representations due to Lindel\"of,
which belong to an attractive but somewhat neglected chapter of complex analysis.
One of the outcomes of such analyses concerns the non-existence of 
linear recurrences with polynomial coefficients annihilating these sequences,
and, accordingly, the non-existence
of linear differential equations with polynomial coefficients
annihilating their generating functions.
In particular, the corresponding generating functions are transcendental.
Asymptotic estimates of certain finite difference sequences come out as a byproduct
of the Lindel\"of  approach.
\end{abstract}

\maketitle

\section*{\bf Introduction}

There  has been recently  a surge  of interest  in methods for proving
that certain   sequences  coming from analysis  or   combinatorics are
\emph{non-holonomic}.     Recall   that    a   sequence~$(f_n)$   is
\emph{holonomic}, or \emph{$P$-recursive},  if  it satisfies  a  linear
recurrence   with coefficients   that  are polynomial   (equivalently,
rational)  in the index~$n$; 
that is,
\[
\sum_{k=0}^d p_k(n) f_{n-k}=0, \qquad
p_k(n)\in \C[n], \qquad p_0 \not\equiv 0.
\]
Put otherwise, its  generating  function
$f(z)$, called \emph{holonomic} or \emph{$D$-finite}, 
satisfies a  linear differential  equation with coefficients  that are
polynomial (equivalently,  rational)   in  the  variable~$z$; that is,
with $f(z):=\sum_{n\ge0} f_n z^n$,
\[
\sum_{k=0}^e q_k(z) \frac{\dd^k}{\dd z^k} f(z)=0 , \qquad
q_k(z)\in \C[z], \qquad q_e \not\equiv 0.
\]
Within
combinatorics, the  holonomic framework has  been largely developed by
Stanley,           Zeilberger,                Lipshitz,            and
Gessel~\cite{Gessel90,Lipshitz88,Lipshitz89,Stanley80,Zeilberger90},
who provided a rich set of closure properties satisfied by the holonomic class.
Since the   class  of  holonomic  functions contains    all  algebraic
functions, establishing  that  a  sequence  is  non-holonomic can   in
particular  be  regarded as  a strong  transcendence  result  for  its
generating function.

In recent years, proofs have appeared of the non-holonomic character of
sequences such as 
\begin{gather*}
\log n,\quad \sqrt{n}, \quad  n^n, \quad \frac{1}{H_n}, 
\quad \varpi_n, \quad \log\log n, \quad \zeta(n),\\
\sqrt{n!},\quad \arctan(n),\quad \sqrt{n^2+1},\quad \e^{\e^{1/n}} 
\end{gather*}
where $H_n$ is a harmonic number, $\varpi_n$ represents the~$n$th prime number, and
$\zeta(s)$ is the Riemann zeta function.
The known proofs range from elementary~\cite{Gerhold04} 
to algebraic and analytic~\cite{BeGeKlLu08,FlGeSa05,Klazar05b,Klazar05}.
The present paper belongs to the category of \emph{complex-analytic} approaches and it
is, to a large extent, a sequel to the paper~\cite{FlGeSa05}. 

Keeping in mind that a univariate holonomic function can have only finitely many singularities, we 
enunciate
the following general principle.
\begin{quote}\small
{\bf Holonomicity criterion.}
The shape of the asymptotic expansion of a holonomic function at a singularity~$z_0$
is strongly constrained, as it can only involve, in sectors of~$\C$, (finite) 
	linear combinations
	of ``elements'' of the form
\begin{equation}\label{struct0}
\exp\left(P(Z^{-1/r})\right)Z^{\alpha} \sum_{j=0}^\infty Q_j(\log Z)Z^{js}
, \qquad Z:=(z-z_0),
\end{equation}
with~$P$ a polynomial, $r$ an integer, $\alpha$ a complex number,
$s$ a rational of~$\Q_{>0}$ and the~$Q_j$  a family of polynomials
of uniformly bounded degree. 
(For an expansion at infinity, change~$Z$ to~$Z:=1/z$.)
\emph{Therefore}, any function whose asymptotic structure at a singularity 
(possibly infinity) is incompatible with
elements of the form~\eqref{struct0}
\emph{must} be non-holonomic.
\end{quote}
Equation~\eqref{struct0} is a 
	 paraphrase of the classical structure theorem for solutions
	of linear differential equations with meromorphic 
	coefficients~\cite{Henrici77,Wasow87}; see also Theorem~2 of~\cite{FlGeSa05}
and the surrounding comments.

\smallskip

What the three of us did in~\cite{FlGeSa05} amounts to implementing the principle above,
in combination with a basic Abelian theorem.
Functional expansions departing from~\eqref{struct0} can then be generated, 
by means of such a
theorem, from corresponding terms in the asymptotic expansions of sequences: for instance,
quantities such as
\[
\log\log n, \quad \frac{1}{\log n}, \quad \sqrt{\log n}, 
\quad \e^{\sqrt{\log n}}, \quad \ldots
\]
constitute forbidden ``elements''  in expansions of holonomic sequences---hence 
their presence immediately betrays
a non-holonomic sequence.
In proving the non-holonomic character of the sequences~$(\log n)$
and~$(\sqrt{n})$, we 
could then simply observe in~\cite{FlGeSa05} that the~$n$th 
order differences (this is a holonomicity-preserving transformation)
involve $\log\log n$ and $1/\log n$ in
an essential way. 

What we do now, is to push the method further, but in another direction, namely, that
of \emph{Lindel\"of representations of generating functions}.
Namely,  for  
a suitable ``coefficient function''~$\phi(s)$, one has
\begin{equation}\label{prelind}
\sum_{n=1}^\infty \phi(n)(-z)^n =-\frac{1}{2\im\pi}\int_{1/2-\im\infty}^{1/2+\im\infty}
\phi(s) z^s \frac{\pi}{\sin \pi s}\, \dd s,
\end{equation}
where the left side is, up to 
an alternating sign, the generating function of the sequence~$(\phi(n))$.
Based on representations of type~\eqref{prelind}, we determine directly the asymptotic behaviour at infinity of 
the generating functions of several sequences given in closed form,
and detect cases that contradict~\eqref{struct0}, hence
entail the non-holonomicity of the sequence~$(\phi(n))$.

Here are typical results that can be obtained by the methods 
we develop. 
\begin{theorem}\label{cor-bell}
$(i)$~The sequences~$\e^{c n^\theta}$, with $c,\theta\in\R$,  are non-holonomic,
except in the trivial cases $c=0$ or $\theta\in\{0,1\}$. In particular,
\begin{equation}\label{exp0}
\e^{\sqrt{n}}, \quad
\e^{-\sqrt{n}}, \quad
\e^{1/n}, \quad
\e^{-1/n}
\end{equation}
are non-holonomic. 

$(ii)$~The sequences
\begin{equation}\label{samp0}
\frac{1}{2^n\pm 1},
\quad \frac{1}{n!+1},
\quad {\Gamma(n\sqrt{2})}, \quad \frac{\Gamma(n\sqrt{2})}{\Gamma(n\sqrt{3})},
\quad \Gamma(n \im),\quad\frac{1}{\zeta(n+2)}
\end{equation}
are also non-holonomic.
\end{theorem}
The non-holonomicity of the sequences
in~\eqref{exp0} also appears in the article 
by Bell \emph{et al.}~\cite{BeGeKlLu08}, 
where it is deduced from an elegant argument involving Carlson's Theorem 
combined with the observation
that $\phi(s)=\e^{c s^\theta}$ is non-analytic and non-polar at~$0$.
The results of~~\cite{BeGeKlLu08} do usually not, however, give access 
to the cases where the function~$\phi(s)$ is either meromorphic throughout~$\C$ or 
entire. In particular, they do not seem to yield the non-holonomic character of
the sequences listed 
in~\eqref{samp0}, except for the first one.
(Indeed, $1/(2^n\pm1)$ has been dealt with in~\cite{BeGeKlLu08}
by an extension of the basic method of that paper.
Moreover, any potential holonomic recurrence for $1/(2^n\pm1)$ can be refuted by
an elementary limit argument, communicated by Fr\'ed\'eric Chyzak and transcribed
in~\cite[Proposition~1.2.2]{Ge05}.)

A great advantage of the Lindel\"of approach
is that it leads to precise asymptotic expansions that are 
of independent interest. 
For instance, in~\S\ref{se:nonhol polar} below we will encounter the expansion
\[
 \sum_{n\geq1} \frac{(-z)^n}{n!+1} \sim -\sum_{k\geq 1} \frac{\pi}{\sin \pi s_k}
 \frac{1}{\Gamma'(s_k+1)} z^{s_k}, \qquad z\to\infty,
\]
where
\[
 (s_k)_{k\geq1} \approx (-3.457, -3.747, -5.039, -5.991, \dots)
\]
are the solutions of $\Gamma(s+1)=-1$. This formula features a remarkable structural
difference to the superficially resembling function $\sum_{n\geq1}(-z)^n/n! = \exp(-z)-1$.

\smallskip\noindent
{\bf Plan of the paper.} The general context of Lindel\"of representations
is introduced in Section~\ref{lindel-sec}; in particular,
Theorem~\ref{lind-thm} of that section provides detailed conditions granting us the validity of~\eqref{prelind}.
As we show next, these representations make it possible to 
analyse the behaviour of generating functions towards~$+\infty$, knowing
growth and singularity properties of the coefficient function~$\phi(s)$
in the complex plane. The global picture is 
a general correspondence of the form:
\begin{center}\em 
\setlength{\tabcolsep}{3pt}   
\begin{tabular}{ccc}
location $s_0$ of singularity of $\phi(s)$
& $\longrightarrow$ & 
singular exponent of $F(z)$ \\
nature of singularity of $\phi(s)$
& $\longrightarrow$ & 
logarithmic singular elements in  $F(z)$.
\end{tabular}
\end{center}
(This is, in a way, a dual situation to singularity analysis~\cite{FlOd90b,FlSe09}.)
Precisely, three major cases are studied here; see Figure~\ref{tafel-fig} below
for a telegraphic summary.
\begin{itemize}
\item[---]
\emph{Polar singularities.} When~$\phi(s)$ can be extended into a meromorphic function, 
the asymptotic expansion of the generating function
of the sequence $(\phi(n))$ at infinity can be read off from the poles of~$\phi(s)$.
Section~\ref{pol-sec} gives a detailed account of the corresponding ``dictionary'',
which is in line with early studies 
by Ford~\cite{Ford60}.
It implies the non-holonomicity of sequences such as
$1/(2^n-1)$, $1/(n!+1)$, $\Gamma(n\sqrt{2})$.
\item[---]
\emph{Algebraic singularities.} 
In this case, a singularity of exponent~$-\lambda$ in the function~$\phi(s)$
essentially induces a term of the form $(\log z)^{\lambda-1}$ in the generating function
as $z\to+\infty$. We show the phenomena at stake by performing a detailed asymptotic 
study of the generating functions of sequences such as~$\e^{\sqrt{n}}$ in Section~\ref{alg-sec},
based on the use of \emph{Hankel contours}.
The non-holonomic character of the sequences~$\e^{\pm n^\theta}$ for~$\theta\in{]0,1[}$
arises as a consequence.
\item[---]
\emph{Essential singularities.}
The case of an essential singularity is illustrated
by $\phi(s)=\e^{\pm 1/s}$: in Section~\ref{sad-sec}, 
we work out the asymptotic form of the generating function at infinity,
based on the \emph{saddle-point method}. In this way, we also obtain the non-holonomic character of sequences
such as~$\e^{\pm 1/n}$ by methods that constitute an alternative to those of~\cite{BeGeKlLu08}.
\end{itemize}

\begin{figure}[t]\small
\begin{center}
\begin{tabular}{ccc}
	\hline\hline \emph{Coefficients} $\phi(n)$ & $z\to\infty$ & $z\to -1$ \\
	\hline 	$\e^{1/n}$ & $-\dfrac{\mathrm{e}^{2\sqrt{\log z}\VPH}}{2\sqrt{\pi} (\log z)^{1/4} }$ & $\dfrac{1}{1+z_\WPH}$ \\
	\hline 	$\e^{-1/n}$ & $-\tfrac{1}{\sqrt{\pi}} (\log z)^{-1/4} \cos\Bigl(2\sqrt{\log z} -\tfrac14 \pi\Bigr)$ & $\dfrac{1 \VPH}{1+z_\WPH}$ \\
	\hline 	$\e^{\sqrt{n}}$ & $-1 - \dfrac{1^\VPH}{\sqrt{\pi \log z}_\WPH}$ & $\dfrac{\sqrt{\pi}\e^{-1/8}}{(1+z)^{3/2}}
	  \exp\Bigl( \dfrac{1}{4(1+z)} \Bigr)$ \\
	\hline 	$\e^{-\sqrt{n}}$ & $-1 + \dfrac{1^\VPH}{\sqrt{\pi \log z}_\WPH}$ & $\E(1)+\E'(1)(1+z)$ \\	\hline\hline
\end{tabular}
\end{center}

\caption{\small Asymptotic forms of~$\E(z;c,\theta)$,
for representative parameter values.}\label{ta:asympt}

\end{figure}

As the discussion above suggests, 
\emph{the present article 
can also serve as a synthetic presentation of
the use of Lindel\"of integrals in the asymptotic analysis of generating 
functions}.
The scope is wide as it concerns 
a large number of generating functions whose coefficients obey an ``analytic law''. This is a subject,
which, to the best of our knowledge, has not been treated systematically in recent decades
(Ford's monograph was published in 1936). In particular,
the joint use of Lindel\"of representations and
of saddle points in Section~\ref{sad-sec}, as well as the corresponding estimates
relative to the family of functions 
\begin{equation}\label{defe}
 \E(z;c,\theta):=\sum_{n=1}^\infty \e^{cn^\theta}(-z)^n,
\end{equation}
appear to be new. Figure~\ref{ta:asympt} summarizes some special cases of the results we obtain
for~$\E(z;c,\theta)$. Note that six essentially different expansions occur, depending
on the parameter values and on the singularity we are interested in,
and observe the surprising occurrence of subtle oscillations associated with $\e^{-1/n}$.

In Section~\ref{dif-sec} we
show that the technology we have developed also  provides 
non-trivial estimates in
the calculus of finite differences. Finally, Section~\ref{se:minus one}
completes our investigation of the function $\E(z;c,\theta)$, by working out its
asymptotic behaviour near $z=-1$.

\section{\bf Lindel\"of Representations}\label{lindel-sec}

Lindel\"of integrals provide  a means to express
a function, knowing an 
``\emph{explicit law}'' for its Taylor coefficients. Let 
$s\mapsto \phi(s)$ be a complex function that is analytic 
at all points of~$\R_{>0}$;
the (ordinary) \emph{generating function} of the sequence of values~$(\phi(n))$ will be taken here
in its alternating form:
\begin{equation}\label{gf0}
F(z):=\sum_{n\ge1} \phi(n) (-z)^n.
\end{equation}
The function~$\phi(s)$, which is typically given by an explicit expression, 
represents the ``law''
of the coefficients of~$F(z)$: it extrapolates the integer-indexed sequence~$(\phi(n))$ 
to a domain of the complex numbers
that must contain the half-line~$\R_{\ge1}$.
The key idea is to introduce the \emph{Lindel\"of integral}
\begin{equation}\label{lind0}
\Lambda(z;\cal C):=\frac{1}{2\im\pi}\int_{\cal C} \phi(s) z^{s} \frac{\pi}{\sin(\pi s)}\, \dd s,
\end{equation}
where~$\cal C$ is a contour enclosing the points $1,2,3,\ldots$ and lying within the domain of analyticity of~$\phi(s)$.
Formally, as well as analytically (see Theorem~\ref{lind-thm} below),
when~$\phi(s)$ is well-behaved near the positive real line,
a basic residue evaluation shows that,
with $\epsilon(\cal C)=\pm1$ representing the orientation of~$\cal C$,
\begin{equation}\label{linda}
\Lambda(z;\cal C) = \epsilon(\cal C)\sum_{n\ge 1} \phi(n) (-z)^{n},
\end{equation}
since the residue of~$\pi/\sin\pi s$ at $s=n$ equals~$(-1)^n$. 
Thus, the Lindel\"of integral~\eqref{lind0} provides a representation of the 
generating function~$F(z)$ of~\eqref{gf0}.
Then, suitable growth conditions on~$\phi(s)$ 
enable us to preserve the validity of~\eqref{linda}
under suitable deformations of the contour~$\cal C$.
We state:

\begin{theorem}[Lindel\"of integral representation]\label{lind-thm}
Let~$\phi(s)$ be a function analytic in~$\Re(s)>0$, satisfying  
the growth condition
\[
\hbox{{\bf (Growth)} \quad
$\ds |\phi(s)|<C\cdot \e^{A|s|}$ as $|s|\to\infty$ 
for some $A\in{]0,\pi[}$ and  $C>0$},
\]
in $\Re(s)\ge1/2$.
Then the generating function $F(z)=\sum_{n\ge1}\phi(n) (-z)^n$ is
analytically continuable to the  sector $-(\pi-A)<\arg(z)<(\pi-A)$, where it 
admits the Lindel\"of representation:
\begin{equation}\label{lind1}
\sum_{n\ge1}\phi(n) (-z)^n
=-\frac{1}{2i \pi} \int_{1/2-\im\infty}^{1/2+\im\infty} \phi(s) z^s \frac{\pi}{\sin\pi s}\, \dd s.
\end{equation}
\end{theorem}
\begin{proof}  By the growth condition, we have $|\phi(n)|=\Oh(\e^{An})$,
so that~$F(z)$ is \emph{a priori} analytic in the open disc $|z|<\e^{-A}$.
The proof proceeds in three moves.

\smallskip

$(i)$~Fix~$z$ to be positive  real and satisfying $z<\e^{-A}$.
Define    the    (positively  oriented)     rectangle   $\cal   R[m,N]$,
with~$m,N\in\Z_{>0}$  by its opposite corners  at $1/2-N\im$ and $m+1/2+N\im$. 
With the notation~\eqref{lind0},
Cauchy's residue theorem provides
\[
\Lambda(z;\cal R[m,N])=\sum_{n=1}^m \phi(n) (-z)^n.
\]
We first dispose of the two horizontal sides of this rectangle.
For~$s$ in the complex plane punctured by small discs of fixed radius
centred at the integers, one has
\begin{equation}\label{boundsin}
\left|\frac{\pi}{\sin\pi s}\right| = \Oh\left(\e^{-\pi |\Im(s)|}\right).
\end{equation}
A consequence of this estimate is that the integrand in the Lindel\"of representation
decays exponentially with~$N$: for fixed~$z\in{]0,\e^{-A}[}$, we have
\[
\phi(s)z^s \frac{\pi}{\sin\pi s}
= \Oh\left(\e^{AN}\e^{-\pi N}\right),
\]
so that the integral along the two horizontal sides of $\cal R[m,N]$
is vanishingly small. One can accordingly let~$N$ tend to~$+\infty$, 
which gives in the limit
\[
\Lambda(z;\cal R[m,\infty])=\sum_{n=1}^m \phi(n) (-z)^n,
\]
meaning that
a partial sum of~$F(z)$ is expressed as the difference of two integrals along
vertical lines.

\smallskip

$(ii)$~We next let~$m$ tend to infinity. With~$z$ still a fixed positive quantity satisfying
$z=\e^{-B}$ for some $B>A$ and $s=m+1/2+\im t$, we have, for some constants~$K,K'$ 
\[
\begin{array}{lll}
\ds \left|\phi(s)z^s \frac{\pi}{\sin\pi s}\right|
&<& \ds K \exp\left[A\sqrt{(m+1/2)^2+t^2}\right] \e^{-Bm} {\e^{-\pi|t|}}\\
&<& \ds K' \e^{(A-B)m} \e^{(A-\pi)|t|}.
\end{array}
\]
Thus, as a function of~$m$, the Lindel\"of
integral is $\Oh(\e^{(A-B)m})$.
This shows that the contribution to the integral~\eqref{lind0}
arising from the rightmost vertical side of $\cal R[m,+\infty]$
is vanishingly small. Letting~$m$ tend to infinity then yields the
representation~\eqref{lind1}
in the limit (upon taking into account the change of sign due to orientation).

\smallskip

$(iii)$~Finally, the convergence of the integral persists,
for any real~$z>0$, given the growth condition on~$\phi$:
this ensures that~$F(z)$ is analytic at all points of the positive real line.
Furthermore, for $z=r \e^{\im\vartheta}$ and $s=1/2+\im t$, we have
\begin{equation}\label{boundzs}
|z^s|=\left|\left(r\e^{\im\vartheta}\right)^{1/2+\im t}\right|=r^{1/2}\e^{-t\vartheta},
\end{equation}
so that the Lindel\"of integral~\eqref{lind1} 
remains convergent in the stated sector, where
it provides the analytic continuation of~$F(z)$.
\end{proof}

This  theorem was familiar  to analysts  about a century  ago: it  forms the
basis   of    Chapter~V  of  Lindel\"of's   treatise~\cite{Lindelof05}
dedicated to ``\emph{prolongement analytique des s\'eries de Taylor}'' and published in 1905;
it   underlies     several    chapters   of   W.   B.    Ford's
monograph~\cite{Ford60}    relative   to     ``\emph{The    asymptotic
developments of functions defined by {Maclaurin} series}'', first published in 1936.
Lindel\"of representations are also central
in several works of  Wright~\cite{Wright40b,Wright41} about generalizations of the exponential 
and Bessel functions.
Last but not least, this circle of ideas
can also provide a basis to Ramanujan's ``Master Theorem''~\cite[pp.~298--323]{Berndt85},
as brilliantly revealed by Hardy in~\cite[Ch.~XI]{Hardy78}. (We propose to return to
properties of the associated ``\emph{magic duality}'' in another study.) 

\section{\bf Sequences with Polar Singularities}\label{pol-sec}

The original purpose of Lindel\"of representations was to provide for 
analytic continuation properties. For instance,
as a consequence of Theorem~\ref{lind-thm},
generalized polylogarithms, such as
\[
\Li_{1/2,0}(z)=\sum_{n\ge 1} \frac{z^n}{\sqrt{n}},
\qquad
\Li_{0,1}(z)=\sum_{n\ge 1} \log n\ z^n,
\]
are continuable into functions analytic in the complex plane slit
along the ray from~$1$ to~$+\infty$;
see for instance~\cite{Flajolet99,Ford60}.
Another fruitful  corollary  of the representations  is 
the possibility of  obtaining \emph{asymptotic expansions}. 
In this section, we examine the simple case
of generating functions whose coefficients admit a meromorphic lifting to~$\C$.

\subsection{Polar singularities}
The following lemma\footnote{%
	This lemma is of course closely related to its specialization 
	to  hypergeometric functions, of which great use
	had been made in early works of Barnes and Mellin; 
	see~ \cite{Slater66} and~\cite[Ch.~XIV]{WhWa27}.
} is used throughout Ford's monograph~\cite[Ch.~1]{Ford60}.

\begin{lemma}[Ford's Lemma, polar case] \label{ford-lem}
Assume that~$\phi(s)$  satisfies the conditions of  Theorem~\ref{lind-thm}.
Assume that it is \emph{meromorphic} in $\Re(s)\ge-B$ and  analytic at
all points of $\Re(s)=-B$. Assume finally that the 
growth condition {\bf (Growth)} extends to the  larger half-plane $\Re(s)\ge -B$. 
Then, the generating function~$F(z)$ admits, as $z\to+\infty$, an asymptotic
expansion of the form
\begin{equation}\label{ford}
F(z)= -\sum_{1/2>\Re(s_0)>-B} \Res\left(\frac{\pi}{\sin \pi s}\phi(s)z^s; s=s_0\right)+
\Oh\left(z^{-B}\right), \qquad z\to+\infty,
\end{equation}
where~$\Res$ is the residue operator 
and the sum comprises all poles of $\phi(s)/\sin \pi s$ 
that lie in the strip $-B<\Re(s)<1/2$.
\end{lemma}
\begin{proof}
Start from~\eqref{lind1},
push the line of integration to the left,  and take residues into
account. This gives directly the expansion~\eqref{ford}. 
\end{proof}
The following observations are to be made concerning~\eqref{ford}.
\begin{itemize}
\item[$(i)$] A pole of $\phi(s)/\sin\pi(s)$ at~$s_0$ and  of order~$\mu\ge 1$ 
gives rise to a residue which is the product of a monomial in~$z$ and a polynomial in~$\log z$:
\[
z^{s_0} P(\log z),\qquad\hbox{where}\quad \deg(P)=\mu-1.
\]
Such poles may arise either from~$\phi(s)$ or from~$1/\sin\pi s$ at $s=0,-1,-2,\cdots$.
In the case where~$\phi(s)$ has no pole at~$s=-n$, the induced residue is
of the form $\phi(-n) z^{-n}$.
\item[$(ii)$] Additional poles of~$\phi(s)$ in the right half-plane can be covered by an easy
extension of the lemma, as long as they have bounded real parts and are not located at integers.
\item[$(iii)$] As a consequence of Item~$(i)$, poles farthest on the right contribute 
the dominant terms
in the asymptotic expansion of~$F(z)$ at~$+\infty$.
\item[$(iv)$] The asymptotic expansions of type~\eqref{ford} hold in
a \emph{sector} containing the positive real line. (To see this, note that the Lindel\"of 
representation remains valid for~$z$ in such a sector and 
that the growth condition in an extended half-plane guarantees the validity of
the residue computation leading to~\eqref{ford}.)
\end{itemize}

\begin{figure}\small
\begin{center}\renewcommand{\arraystretch}{1.25}\renewcommand{\tabcolsep}{-1.5pt}
\begin{tabular}{lcc}
\hline\hline
\emph{Point~$s_0$} & $\phi(s)$ & $F(z)$ \\
\hline
{\bf Regular point [\S\ref{lindel-sec}]} && \\
$s_0=-n$, $\phi(s)$ analytic & $\phi(-n)+\cdots$ & $(-1)^{n+1}\phi(-n) z^{-n}$ \\
\hline
{\bf Polar singularity [\S\ref{pol-sec}] }&& \\
$s_0\not\in\Z$, simple pole of~$\phi(s)$ & $\ds\frac{1}{s-s_0}$ & 
$\ds -\frac{\pi}{\sin\pi s_0} z^{s_0}$ \\
$s_0$ pole of order~$\mu$ of $\phi(s)/\sin\pi s$ &
& $\ds z^{s_0}P_{\mu-1}(\log z)$\\
\hline
{\bf Algebraic singularity [\S\ref{alg-sec}] }\\
$s_0\notin \Z$ & 
$\ds\frac{1}{(s-s_0)^\lambda}$ & $\ds -\frac{\pi}{\sin\pi s_0} \frac{z^{s_0}(\log z)^{\lambda-1}}{\Gamma(\lambda)} $
\\
\hline
{\bf Essential singularity [\S\ref{sad-sec}]} \\
$s_0\in \C$, $\theta<0$  & $\ds \e^{+(s-s_0)^\theta}$ & $\ds -K_1z^{s_0}(\log z)^{\frac{\theta}{2(1-\theta)}}
\exp(K_2(\log z)^{\frac{\theta}{\theta-1}})$ \\
$s_0\in \C$, $\theta<0$  & $\ds \e^{-(s-s_0)^\theta}$ & $\ds L_1z^{s_0}(\log z)^{\frac{\theta}{2(1-\theta)}}
\exp((L_2+\im L_3)(\log z)^{\frac{\theta}{\theta-1}})$ 
\\[2ex]
\hline\hline
\end{tabular}
\end{center}
\caption{\label{tafel-fig}\small Sample cases of the 
correspondence between local (regular or singular) elements of a function
at a point~$s_0$ and the main asymptotic term in the expansion of
the generating function~$F(z)$ at infinity.}
\end{figure}

\subsection{Non-Holonomicity Resulting from Polar Singularities}\label{se:nonhol polar}

We list now a few sequences that may be proved 
non-holonomic by means of Ford's lemma (Lemma~\ref{ford-lem}) in conjunction
with the non-holonomicity criterion based on~\eqref{struct0}.
We restrict ourselves to
prototypes; a large number of variations are clearly possible.

\begin{proposition} 
The following sequences 
are non-holonomic (with~$\im=\sqrt{-1}$):
\begin{equation}\label{potpourri}\renewcommand{\arraystretch}{1.4}
\left\{
\begin{array}{lll}
\ds f_{1,n}=\frac{1}{1+n!} ,
& \ds f_{2,n}=\Gamma( n\sqrt{2}),
& \ds f_{3,n}=\frac{\Gamma(n\sqrt{2})}{\Gamma(n\sqrt{3})}\\
\ds f_{4,n}=\frac{1}{2^n-1},
& \ds f_{5,n}=\Gamma(n\im),
&\ds f_{6,n}=\frac{1}{\zeta(n+2)}.
\end{array}\right.
\end{equation}
\end{proposition}
\begin{proof}
$(i)$ First the sequences $f_{1,n},f_{2,n},f_{3,n}$ 
are treated as direct consequences of Lemma~\ref{ford-lem}.

For $f_{1,n}=1/(1+n!)$,  we observe that
the extrapolating function $\phi(s)=1/(1+\Gamma(s))$
is meromorphic in the whole of~$\C$.
Thus the basic argument of~\cite[Th.~7]{BeGeKlLu08}
is not applicable.
However, examination of the roots of the equation $\Gamma(s)=-1$ reveals that 
there are roots near 
\[
\hbox{\small -2.457024, ~~-2.747682,~~-4.039361,~~-4.991544,~~-6.001385,~~-6.999801,
~~-8.000024, ~~-8.999997,} 
\]
and so on, in a way that precludes the possibility of these roots
to be accommodated into a finite number of arithmetic progressions.
(To see this, note that if $\Gamma(s)=-1$, then $\Gamma(1-s)=-\pi/\sin\pi s$ by Euler's
reflection formula. As~$s$ moves farther to the left, the quantity
 $\Gamma(1-s)$ becomes very large; thus $\sin\pi s$ must be extremely close to zero.
Hence~$s$ itself must differ from an integer~$-k$ by a very small quantity,
which is found to be ~$\sim(-1)^{k-1}/k!$, in accordance with the numerical data above.)
When transposed to the asymptotic expansion of~$F(z)$ at infinity
by means of Ford's Lemma (Lemma~\ref{ford-lem}), this 
can be recognized to contradict the holonomicity criterion~\eqref{struct0}.
Indeed, 
the exponents of $Z=z-z_0$ that can occur in the singular expansion 
of a holonomic function are invariably to be found amongst a \emph{finite} union of 
arithmetic progressions, each of whose common difference must be a rational number.

A similar reasoning applies to $f_{2,n}=\Gamma(n\sqrt{2})$, 
$f_{3,n}=\Gamma(n\sqrt{2})/\Gamma(n\sqrt{3})$, and more generally\footnote{
	The corresponding generating functions are related to classical Mittag-Leffler 
	and Wright functions~\cite{Wright40b, Wright40, Wright41}. They arise for instance 
	in fractional evolution equations~\cite{MaGo00}
	and in the stable laws of probability theory~\cite[\S XVII.6]{Feller71}.}
$\Gamma(\alpha n)$, where $\alpha\in\R\setminus\Q$. For instance, in the case of $f_{2,n}$,
we should consider~$\phi(s)=\Gamma(s\sqrt{2})/\Gamma(s)^2$, where the normalization
by~$\Gamma(s)^2$ ensures both the analyticity at~$0$ of~$F(z)$ and the
required growth conditions at infinity of Lemma~\ref{ford-lem}. 
The poles of~$\phi(s)$ are now nicely aligned horizontally 
in a single arithmetic progression, but 
their common difference $(1/\sqrt{2})$ is an irrational number.

$(ii)$ Next, 
for the sequences $f_{4,n},f_{5,n},f_{6,n}$, we can recycle the proof technique of Lemma~\ref{ford-lem},
so as to allow for an infinity of poles in a fixed-width vertical strip
(details omitted). 

In the case of~$f_{4,n}$, we are dealing with the function $\phi(s)=1/(2^s-1)$,
which is meromorphic in~$\C$, and has infinitely many regularly spaced poles
at $s=2\im k\pi/\log2$, with $k\in \Z$. The ``dictionary'' suggested by Figure~\ref{tafel-fig}
and Lemma~\ref{ford-lem}
applies to the effect that~$F(z)$ has an expansion involving \emph{infinitely many}
elements of the form  $z^{2\im k\pi/\log 2}$. 
A similar argument applies to $f_{5,n}=\Gamma(\im n)$ which now has a half line
of regularly spaced, vertically aligned, poles on~$\Re(s)=0$.

As another consequence 
(but not a surprise!), 
the sequence $f_{6,n}=1/\zeta(n+2)$ is non-holonomic, since 
$\phi(s)=1/\zeta(s+2)$ satisfies the growth conditions of Theorem~\ref{lind-thm} \emph{and}
the Riemann zeta function has infinitely many non-trivial zeros.
Likewise, for $f_{4,n}$ and $f_{5,n}$, the expansion of~$F(z)$ contains exponents with
infinitely many distinct imaginary parts.
(Non-holonomicity does \emph{not} depend on the Riemann hypothesis.)
\end{proof}

In summary, assuming meromorphicity of the coefficient function~$\phi(s)$ in~$\C$
accompanied by suitable growth conditions in half-planes, sequences of the form~$(\phi(n))$ are 
bound to be non-holonomic as soon as the set of
of poles of~$\phi(s)$ 
is \emph{not} included in a finite union of arithmetic progressions
with rational common differences.
In a way this extends the results of~\cite{BeGr93,BeGr96} to cases of functions that are not entire.

\section{\bf Sequences with Algebraic  Singularities}\label{alg-sec}
\subsection{Analogue of Ford's Lemma}
Ford    also     discusses    the    case      where~$\phi(s)$     has
algebraic singularities. In this situation,  \emph{it  is no longer
possible   to move the contour of integration  past 
singularities;  Hankel   contour integrals  replace the
residues of~\eqref{ford},   and only  expansions in descending  powers
of~$\log   z$ (rather   than~$z$) are obtained}~\cite[Ch.~III]{Ford60}.
See Figure~\ref{tafel-fig} for an aper{\c c}u. 
We first define the kind of singularities that can be handled.
\begin{definition}
	The function~$\phi(s)$ is said to have a singularity of \emph{algebraic type}
$(\lambda,\theta,\psi)$ at~$s_0$ if a local expansion of the form
 \begin{equation}\label{eq:alg loc exp}
   \phi(s) = (s-s_0)^{-\lambda} \psi((s-s_0)^{\theta})
 \end{equation}
holds in a slit neighborhood of~$s_0$, where $\lambda\in\C$, $\Re(\theta)>0$, and
\begin{equation}\label{eq:psi exp}
   \psi(s) = \sum_{k\geq0} p_{k} s^k
\end{equation}
 is analytic at zero.
\end{definition}

To state the full result, we denote by~$b_j(s_0)$ the coefficients in the expansion
\begin{equation}\label{eq:sin exp}
 \frac{\pi}{\sin\pi s} = \sum_{j\geq -1}b_j(s_0) (s-s_0)^j, \qquad s_0\in\C.
\end{equation}
Note that $b_{-1}(s_0)=0$ for $s_0\notin\Z$, while for~$n\in\Z$, we have
\begin{equation}\label{b_j(0)}
b_{-1}(n)=(-1)^n,\qquad b_{2k}(n)=0,\qquad b_{2k-1}(n)=(-1)^n(2-2^{2-2k})\zeta(2k).
\end{equation}

\begin{lemma}[Ford's Lemma, algebraic case]\label{le:alg}
 Suppose that~$\phi(s)$ is analytic throughout~$\C$, except for finitely many
 singularities of algebraic type~$(\lambda_i,\theta_i,\psi_i)$ at $s_i$, $i=1,\dots,M$. The $m$-th branch cut should be at
 an angle $\omega_m\in {]{-\tfrac12\pi},0[} \cup {]0,\tfrac12 \pi[}$ with the negative real axis,
 and the cuts may not intersect each other or the set~$\Z_{>0}$.
 We impose the growth condition {\bf (Growth)} with
 \[
   A < \pi \min_{1\leq m\leq M}|\sin \omega_m|.
 \]
 Assume furthermore that the singularities~$s_m$ are sorted so that $\Re(s_1)=\dots=\Re(s_N)>\Re(s_{N+1})\ge\dotsb$, for some $1\leq N\leq M$.
 Then~$F(z)$ has the asymptotic expansion
 \begin{multline}\label{eq:alg exp}
   F(z) \sim \sum_{0\leq n\le n_{\mathrm{max}}} (-1)^{n+1}\phi(-n) z^{-n} \\
    - \sum_{m=1}^N z^{s_m} \sum_{\substack{k\geq0 \\ j\geq -1}} \frac{p_{m,k}
b_j(s_m)}{\Gamma(-\theta_m k-j+\lambda_m)}
    (\log z)^{-\theta_m k-j+\lambda_m-1},
 \end{multline}
where the $p_{m,k}$ are the coefficients in the expansion of~$\psi_m$ at~$s_m$,
in the sense of~\eqref{eq:psi exp}, and
the summation range of the first sum in~\eqref{eq:alg exp} is determined by
 \begin{equation}\label{eq:nmax}
   n_{\mathrm{max}}=
   \begin{cases}
     -\Re(s_1) & \Re(s_1)\in\Z\ \text{and}\ s_m\notin\Z\ \text{for}\ 1\leq m\leq N  \\
     \lceil -\Re(s_1)\rceil -1 & \text{otherwise}.
   \end{cases}
 \end{equation}
 The variable~$z$ may tend to infinity in any sector with vertex at zero that avoids the negative real axis.
\end{lemma}
\begin{proof}
 The statement assembles and generalizes results by Barnes~\cite{Barnes06} and Ford~\cite{Ford60},
 who treat the case~$\theta_m=1$ (the latter reference offers a very detailed discussion).  
 Extending the integration contour as usual (see Figure~\ref{fig:hankel}), we find the expansion
 \begin{equation}\label{eq:alg F}
   F(z) + \frac{1}{2\im\pi} \sum_{m=1}^M \int_{\mathcal{H}_m} \phi(s) z^s \frac{\pi}{\sin\pi s} \dd s
   \quad \sim \!\!\!\!\!
\sum_{\substack{n\geq 0 \\ -n \notin \{s_1,\dots,s_M\}}} \!\!\!\! (-1)^{n+1} \phi(-n)z^{-n},
 \end{equation}
 where~$\mathcal{H}_m$ is a narrow Hankel-type contour, positively oriented and 
 embracing the $m$-th branch cut.
 \begin{figure}[t]
 \begin{center}
   \setlength{\unitlength}{1truecm}
   \begin{picture}(3.5,7.3)
     \put(0,0){\includegraphics[width=5truecm]{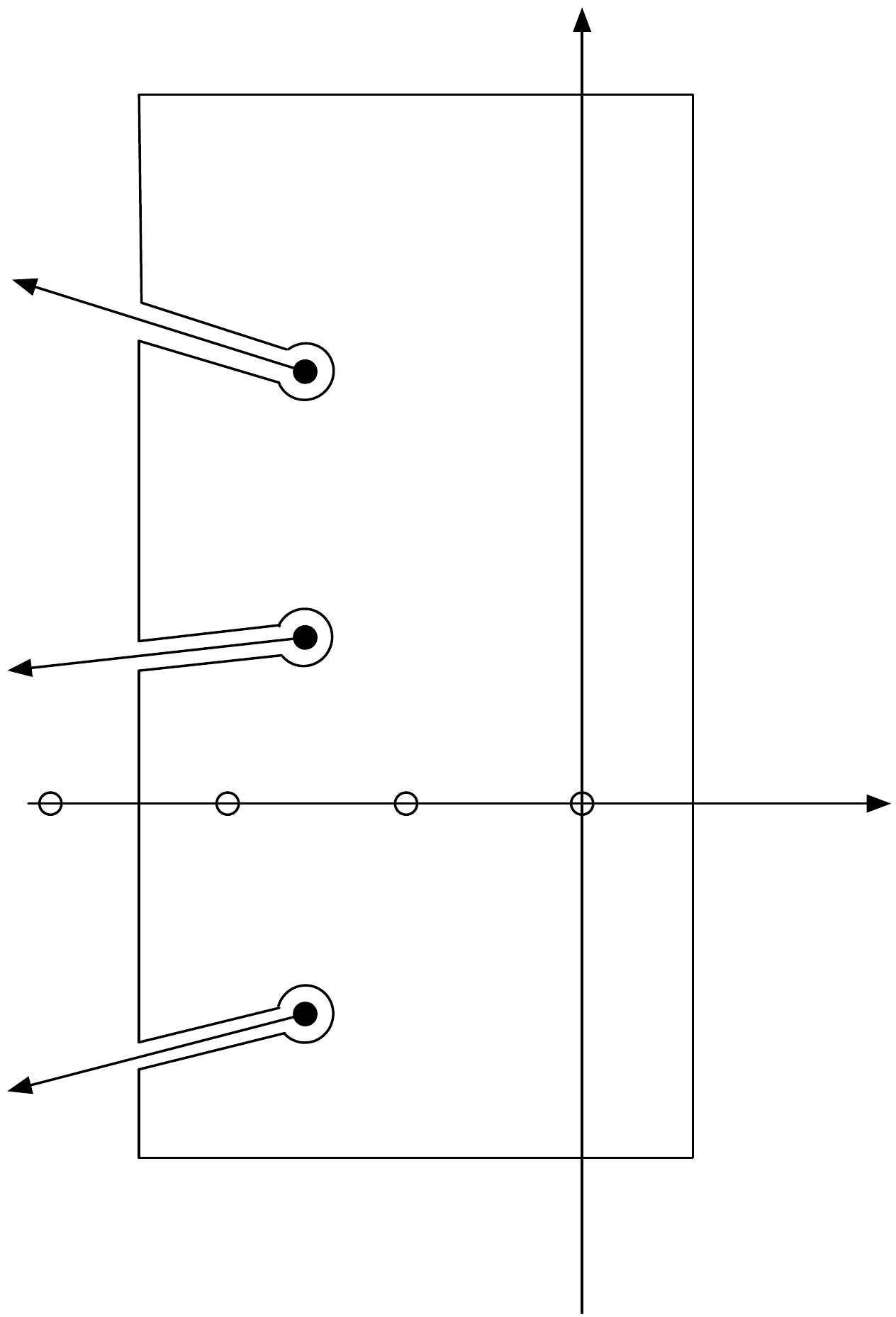}}
     \put(1.95,1.61){$s_1$}
     \put(1.95,3.72){$s_2$}
     \put(1.95,5.2){$s_3$}
     \put(2.95,2.5){$0$}
     \put(2.0,2.5){$-1$}
     \put(1.0,2.5){$-2$}
     \put(0.0,2.5){$-3$}
   \end{picture}
   \caption{A rectangular integration contour embracing branch cut singularities.}
   \label{fig:hankel}
 \end{center}
 \end{figure}
 Suppose first that $M=1$, $s_1=0$, set $b_j:=b_j(0)$, $p_k:=p_{1,k}$ 
 (see~\eqref{eq:sin exp} and~\eqref{eq:psi exp}),
 and drop the index of~$\psi_1,\lambda_1$, etc. In a slit neighborhood of zero, we then
 have
 \[
   \phi(s) \frac{\pi}{\sin\pi s} = s^{-\lambda} \sum_{\substack{k\geq0 \\ j\geq -1}} p_k b_j s^{\theta k+j}.
 \]
 We first deal with a truncation
 \[
   T_K(s):=s^{-\lambda} \sum_{\substack{k\geq0, j\geq -1 \\ \theta k+j <K}} p_k b_j s^{\theta k+j}
 \]
 of this expansion. Plugging $T_K$ in~\eqref{eq:alg F} and substituting $y=-s\log z$, we then
 evaluate the integral termwise by Hankel's formula for the Gamma function:
 \[
   -\frac{1}{2\im\pi} \int_{-\mathcal{H}}\e^{-y}(-y)^\eta \dd y = \frac{1}{\Gamma(-\eta)}, \qquad \eta\in\C.  \]
 We thus obtain
 \begin{align*}
   \frac{1}{2\im\pi}  \int_{\mathcal{H}} z^s T_K(s) \dd s 
   &= -\sum \frac{p_kb_j}{2\im\pi} (\log z)^{-\theta k-j+\lambda-1} \int_{-(\log z)\mathcal{H}} \e^{-y}(-y)^{\theta k+j-\lambda}dy \\
   &= \sum_{\substack{k\geq0, j\geq -1 \\ \theta k+j <K}} \frac{p_{k}b_j}{\Gamma(-\theta k-j+\lambda)}
    (\log z)^{-\theta k-j+\lambda-1}.
 \end{align*}
 We now have to show that
 \begin{equation}\label{eq: alg est}
   \int_{\mathcal{H}} z^s \Bigl(\phi(s) \frac{\pi}{\sin\pi s} - T_K(s) \Bigr) \dd s = \Oh((\log z)^{-K+\lambda-1}).
 \end{equation}
 For~$s$ inside the circle of convergence of~$\psi$, we can expand $\phi(s)\pi/\sin\pi s$,
 substitute $y=-s\log z$, and estimate the series tail $\sum_{\theta k+j\geq K}$ by the triangle inequality.

 To treat the remaining part of the contour~$\mathcal{H}$, we pick a point~$w$ on either
 of the rectilinear portions of~$\mathcal{H}$, and lying inside the circle of convergence of~$\psi$. By taking~$\mathcal{H}$
 narrow enough, both rays of~$\mathcal{H}$ will admit such a point with $\Re(w)<0$.
 Then the portion $\int_{w}^{-\infty}$ of the integral in~\eqref{eq: alg est} is smaller than
 \begin{equation}\label{eq:alg est2}
   \e^{\Re(w\log z)} \int_{w}^{-\infty} \e^{\Re((s-w)\log z)} \Bigl(C 
    +|s|^{|\lambda|} \sum_{\theta k+j <K} |p_k b_j| |s|^{\theta k+j} \Bigr) \dd s,
 \end{equation}
 where the boundedness of $|\phi(s)\pi/\sin\pi s|<C$ results from our growth assumption on~$\phi(s)$ and the
 fact that $\sin \pi s = \Oh(\exp(-\pi|\Im(s)|))$. The integral in~\eqref{eq:alg est2} converges,
 and
 \[
   \e^{\Re(w\log z)} = \Oh(|z|^{\Re(w)}),
 \]
 which is negligible in the logarithmic scale of the problem, as $\Re(w)<0$.

 This completes the proof in the special case $M=1$, $s_1=0$. The full result follows from the special case
 by performing the substitutions $s\mapsto s+s_m$ in~\eqref{eq:alg F}.
 Deleting redundant terms in the resulting expansion then yields~\eqref{eq:alg exp}. Note that in the
 first case of~\eqref{eq:nmax} we must include the contribution of the pole at $s=\Re(s_1)$ in~\eqref{eq:alg exp},
 which gives rise to the summand $n=-\Re(s_1)$. Otherwise, we need to consider only the poles whose
 real part is larger than~$\Re(s_1)$. 
\end{proof}

We make the following remarks concerning Lemma~\ref{le:alg}.
\begin{itemize}
 \item[$(i)$] Clearly, the 
statement extends to functions~$\phi(s)$ having both poles and  algebraic singularities.
   In fact the expansion~\eqref{eq:alg exp} essentially remains  valid then, as the
   reciprocal of the Gamma function vanishes at the non-positive integers.
   For instance, if $s_m\notin\Z$ is a simple
   pole, then only the summand $k=j=0$ of the inner sum remains, which is in line with Lemma~\ref{ford-lem}.
   The dominating singularities $s_1,\dots,s_N$ that enter the expansion~\eqref{eq:alg exp} must then not only comprise
   the rightmost algebraic singularities, but also the poles whose real parts are equal to theirs or greater.
 \item[$(ii)$] We disallow horizontal branch cuts in the lemma, in order to take advantage of the exponential
   decrease of $\pi/\sin \pi s$ along vertical lines. If a horizontal cut is present, and~$\phi(s)$ stays
   bounded near the cut, the result persists.
 \item[$(iii)$] There is a slight error in the statement of~\cite[\S5]{Ford60}: Ford assumes that his function~$P$, which
   corresponds to our $\phi(-s)\pi s/\sin \pi s$, is bounded in a right half-plane. This is usually too restrictive,
   due to the poles of $1/\sin \pi s$, and is in fact not satisfied by the application in~\cite[\S6]{Ford60}.
 \item[$(iv)$] By putting~$\phi(s)=s^{-\lambda}$ in Lemma~\ref{le:alg}, we recover the classical expansion
   of the polylogarithm function at infinity~\cite{Ford60,Pi68}.
\end{itemize}

\subsection{Asymptotic Analysis of the Generalized Exponential $\E(z;c,\theta)$ when $\theta\in{]0,1[}$.}\label{algforexp}
Recall that in this case~$\phi(s)=\exp(cs^\theta)$, so that it is a straightforward application of Lemma~\ref{le:alg}, with $M=1$, $\lambda=0$, and $\psi(s)=\e^{cs}$.
In fact this example was our initial motivation to extend Ford's result to the case where~$\theta_m\neq1$.
As for the branch cut of the function~$s^\theta$, we may put it at any direction allowed by Lemma~\ref{le:alg}.
The resulting expansion is
\begin{equation}\label{eq:hankel expans}
 \E(z;c,\theta) \sim  - \sum_{\substack{k\geq 0\\j\geq -1}}
   \frac{c^k b_j(0)}{k!\Gamma(-k\theta-j)} (\log z)^{-k\theta-j-1}, \qquad z\to\infty,
\end{equation}
with~$b_j(0)$ as given in~\eqref{b_j(0)}.
In particular, for the coefficient sequences~$\e^{\pm\sqrt{n}}$ we obtain
\[
 \E(z;c=\pm 1,\theta=\tfrac12) = -1 \mp \frac{1}{\sqrt{\pi \log z}} + \Oh\left(\frac{1}{(\log z)^{3/2}}\right).
\]

\subsection{Non-Holonomicity Resulting from Algebraic Singularities}

The estimates of Equation~\eqref{eq:hankel expans}, when compared with the holonomicity criterion~\eqref{struct0}, 
immediately yield 
the non-holonomic character of simple sequences 
involving the exponential function, such as $\e^{\pm\sqrt{n}}$ in Theorem~\ref{cor-bell}.
More generally, we can state the following result.

\begin{proposition}
 Suppose that~$\phi(s)$ satisfies the assumptions of Lemma~\ref{le:alg}, and has a non-polar singularity
 at~$s_1$ with an expansion of the type~\eqref{eq:alg loc exp}. Then the sequence~$(\phi(n))_{n\geq1}$ is not
 holonomic.
\end{proposition}
\begin{proof}
 This follows readily from the expansion~\eqref{eq:alg exp} and the holonomicity criterion~\eqref{struct0}.
 Without loss of generality, we assume that~$s_1$ has maximal real part among the non-polar singularities
 of~$\phi(s)$. Now choose~$k_0\geq0$ such that~$p_{k_0}\neq0$ and $k_0\theta_1-\lambda_1\notin\Z$.
 (If there was no such~$k_0$, then~$s_1$ would either be a pole or no singularity at all.)
 It then suffices to pick some~$j_0\geq-1$ with~$b_{j_0}\neq0$ to exhibit a logarithmic term with ``forbidden'',
 non-integral exponent and non-zero coefficient in~\eqref{eq:alg exp}.
\end{proof}

Note that this proposition could also have been proved by the method
of~\cite{BeGeKlLu08}, based on Carlson's Theorem, but without the ``constructive'' 
feature of obtaining an asymptotic expansion of the associated generating functions.

\section{\bf Sequences with Essential Singularities}\label{sad-sec}

Going beyond the topics covered by Ford's treatise~\cite{Ford60}, we now investigate cases
where the coefficient function~$\phi(s)$ in~\eqref{prelind} has essential singularities.
We do not aim at a general statement here, but instead restrict
attention to the generalized exponential $\E(z;c,\theta)$
from~\eqref{defe}, with $\theta<0$.
(In fact a general result encompassing
both Theorems~\ref{thm:saddle} and~\ref{thm:twosaddle} below would probably be
rather unwieldy.)
This illustrates the use of Lindel\"of representations in conjunction with the
saddle-point method~\cite{deBruijn81,FlSe09}.
The resulting asymptotic formulas (Theorems~\ref{thm:saddle} and~\ref{thm:twosaddle} below)
will be immediately recognized to 
be incompatible with the structure formula of~\eqref{struct0}:
in this way, the present section completes our proof of Theorem~\ref{cor-bell}.
(Two rather easy supplementary arguments, which 
serve to cover the whole range of parameter values,
but involve only crude asymptotic analysis, are collected in Subsection~\ref{se:thm compl}.)

\subsection{Asymptotic Analysis of the Generalized Exponential $\E(z;c,\theta)$
when $c>0$ and $\theta<0$.}\label{se:saddle}

In this section we determine the asymptotic behaviour of~$\E(z)$ near infinity for
positive~$c$ and negative~$\theta$.
We present the analysis in the special case $c=1$, $\theta = -1$. The generalization to
arbitrary $c>0$ and $\theta<0$ is then easy.
We start once more from the Lindel\"of integral representation
\begin{equation}\label{eq:lind}
 \E(z;1,-1) = -\frac{1}{2\im\pi} \int_{1/2-\im\infty}^{1/2+\im\infty} \e^{1/s}z^s \frac{\pi}{\sin\pi s} \dd s.
\end{equation}
Neglecting the effect of $\pi/\sin\pi s$, the derivative
\[
 \frac{\partial}{\partial s} \e^{1/s}z^s = (\log z - s^{-2})\e^{1/s}z^s
\]
reveals a saddle point near $s=L^{-1/2}$, where $L:=\log|z|$.
We accordingly move the integration contour in~\eqref{eq:lind} to the left, obtaining
\begin{equation}\label{eq:saddle h}
 \E(z;1,-1) = -\frac{1}{2\im\pi} \int_{L^{-1/2}-\im\infty}^{L^{-1/2}+\im\infty} \e^{1/s}z^s \frac{\pi}{\sin\pi s} \dd s.
\end{equation}
We set $s=L^{-1/2}+\im t$.
The main contribution to the integral arises near $t=0$, say for $|t|<L^{-\alpha}$, where it will turn out
that $\tfrac23 < \alpha < \tfrac34$ is a good choice. The expansions which we require to approximate the
integrand around the saddle point are collected in Figure~\ref{ta:expansions}. (We will recycle them in the
next section.)
\begin{figure}[t]
\begin{center}\renewcommand{\arraystretch}{1.25}
\begin{tabular}{l}
	\hline\hline $s=aL^{-1/2}+b t, \qquad a,b \in \mathbb{C}\setminus \{0\}, \quad |t| < L^{-\alpha}, \quad \tfrac23<\alpha<\tfrac34$ \\
	\hline $\pm \dfrac{1 \VPH}{s \WPH} = \pm \tfrac{1}{a}L^{1/2} \pm \tfrac{b^2}{a^3}L^{3/2}t^2 \mp \frac{b}{a^2}Lt + \Oh(L^{2-3\alpha})$ \\
	\hline $\exp(\pm \dfrac{1 \VPH}{s \WPH}) = \exp(\pm \tfrac{1}{a}L^{1/2} \pm \tfrac{b^2}{a^3}L^{3/2}t^2 \mp \frac{b}{a^2}Lt) (1+\Oh(L^{2-3\alpha}))$ \\
	\hline $\dfrac{\pi \VPH}{\sin \pi s \WPH} = \frac{1}{a}L^{1/2}(1+\Oh(L^{1/2-\alpha}))$ \\
	\hline $z^s = \exp(aL^{1/2}+bLt)(1+\Oh(L^{-1/2}))$ \\
	\hline\hline
\end{tabular} 
\end{center}

\caption{\label{ta:expansions}
Four elementary asymptotic expansions. The variable~$L$ tends to~$+\infty$, and the first line specifies the
range of~$s$ and the fixed parameters~$a,b,$ and~$\alpha$.} 

\end{figure}
{}From these we obtain, provided that~$\alpha>\tfrac23$, the approximation
\begin{equation*}
 \left.\frac{\e^{1/s}z^s}{\sin \pi s}\right|_{s=L^{-1/2}+\im t} = \tfrac{1}{\pi} L^{1/2} \exp(2L^{1/2}-L^{3/2}t^2) \cdot (1 + \Oh(L^{2-3\alpha})).
\end{equation*}
This puts us in a position to evaluate the central part of the integral~\eqref{eq:saddle h}:
\begin{align}
 -\frac{1}{2\im}\int_{L^{-1/2}-\im L^{-\alpha}}^{L^{-1/2}+\im L^{-\alpha}} & \frac{\e^{1/s}z^s}{\sin \pi s}\dd s =
  -\frac{L^{1/2}\mathrm{e}^{2L^{1/2}}}{2\pi} \int_{-L^{-\alpha}}^{L^{-\alpha}}
  \mathrm{e}^{-L^{3/2}t^2}\mathrm{d}t \notag\\
 &= -\frac{L^{1/2}\mathrm{e}^{2L^{1/2}}}{2\pi} 
  \int_{-\sqrt{2}L^{3/4-\alpha}}^{\sqrt{2}L^{3/4-\alpha}} \tfrac{1}{\sqrt{2}}L^{-3/4}
  \mathrm{e}^{-r^2/2} \mathrm{d}r \notag\\
 &\sim -\frac{\mathrm{e}^{2L^{1/2}}}{2\sqrt{2}\pi L^{1/4}}
  \int_{-\infty}^\infty \mathrm{e}^{-r^2/2} \mathrm{d}r = -\frac{\mathrm{e}^{2L^{1/2}}}{2\sqrt{\pi} L^{1/4}} , \label{eq:central part}
\end{align}
with a relative error of~$\mathrm{O}(L^{2-3\alpha})$.
In order to let the integration bounds of the Gaussian integral tend to infinity, we have
assumed $\alpha<\tfrac34$ here. The tails of the Gaussian integral then decrease exponentially in~$L$.

In order to show that this is indeed the dominant part of the integral~\eqref{eq:saddle h},
it remains to prove that the portion of the integral from $\im L^{-\alpha}$ to~$\im\infty$ (and thus, by symmetry,
also from~$-\im\infty$ to~$-\im L^{-\alpha}$) grows more slowly.

First we consider $t=\Im(s)\geq 1$. In this range we have $\e^{1/s}=\Oh(1)$
, and
\[
 |z^s|=\exp(L^{1/2}-t\arg z) \qquad \text{and} \qquad 1/\sin\pi s = \Oh(\e^{-\pi t})
\]
lead to the bound $\exp(L^{1/2})\cdot\int_1^\infty \exp(-(\pi+\arg z)t) \dd t$.
To make the integral convergent, we assume that~$z$ tends to infinity in a sector
that does not contain the negative real axis.

Now consider $L^{-\alpha}\leq t< 1$. The factor~$z^s$ is $\Oh(\exp(L^{1/2}))$ there,
and the estimate
\[
 |\e^{1/s}| \leq \e^{1/|s|} =\Oh(\exp(L^{1/2}-\tfrac12 L^{3/2-2\alpha}))
\]
follows from evaluating~$\e^{1/|s|}$, which is a decreasing function of~$\Im(s)$,
at $s=L^{-1/2}+\im L^{-\alpha}$. Since
\[
 1/\sin \pi s = \Oh(1/s) = \Oh(L^\alpha),
\]
we have established the tail estimate
\[
 \int_{L^{-1/2}+\im L^{-\alpha}}^{L^{-1/2}+\im\infty} \e^{1/s}z^s \frac{\pi}{\sin\pi s} \dd s = \Oh( L^\alpha \cdot\exp(2L^{1/2}-\tfrac12 L^{3/2-2\alpha})),
\]
which grows slower than the absolute error in~\eqref{eq:central part}.
Hence the asymptotic behaviour of~$\E(z;1,-1)$ near infinity is
\[
 \E(z;1,-1) = -\frac{\mathrm{e}^{2\sqrt{\log z}}}{2\sqrt{\pi} (\log z)^{1/4} }
  \left( 1 + \mathrm{O}((\log z)^{-1/4+\varepsilon}) \right),\qquad z\to\infty.
\]
The error term follows from taking $\alpha\in{]\tfrac23,\tfrac34[}$ close to~$\tfrac34$.

All steps of the previous derivation 
are easily extended, when $1/s$ is replaced by~$c s^\theta$, 
which yields the following result.

\begin{theorem}\label{thm:saddle}
 Let $c>0$ and $\theta<0$ be real numbers. Then
 \begin{equation}\label{eq:saddle exp}
   \E(z;c,\theta) = -K_1 (\log z)^{\frac{\theta}{2(1-\theta)}}
    \exp(K_2 (\log z)^{\frac{\theta}{\theta-1}} )
    \left(1 + \Oh\left(\left(\log z\right)^{-\mu}\right) \right)
 \end{equation}
 as $z\to \infty$ in an arbitrary sector with vertex at zero that does not contain the negative real axis.
 The positive constants~$K_1$ and~$K_2$ are defined by
 \[
   K_1:= (2\pi(1-\theta))^{-1/2} (-c\theta)^{\frac{1}{2(\theta-1)}} \quad
    \text{and} \quad K_2 :=\left(1 - \tfrac{1}{c\theta} \right)(-c\theta)^{\frac{1}{1-\theta}},
 \]
 and the exponent~$\mu$ of the relative error estimate is
 \[
   \mu :=
   \begin{cases}
     \tfrac{\theta}{2(\theta-1)} - \varepsilon & \theta \geq -2 \\
     \tfrac{1}{1-\theta} & \theta < -2,
   \end{cases}
 \]
 with~$\varepsilon$ an arbitrary positive real.
\end{theorem}

\subsection{Asymptotic Analysis of the Generalized Exponential $\E(z;c,\theta)$,
when $c<0$ and $\theta<0$.}\label{se:two saddle}

We present the detailed proof for the parameter values $c=\theta=-1$. Then,
 the integrand of the Lindel\"of integral
\begin{equation}\label{eq:lind2}
 \E(z;-1,-1) = -\frac{1}{2\im\pi} \int_{1/2-\im\infty}^{1/2+\im\infty} \e^{-1/s}z^s \frac{\pi}{\sin\pi s} \dd s
\end{equation}
has two saddle points, at~$\pm \im L^{-1/2}$, roughly, which will induce an oscillating factor.
The argument of the axis of the upper saddle point
is~\cite{deBruijn81}
\[
 \frac{\pi}{2} - \frac{1}{2} \arg \left.\frac{\dd^2}{\dd s^2}(-1/s+Ls)\right|_{s=\im L^{-1/2}} = \frac{3\pi}{4},
\]
and that of the lower saddle point is~$\tfrac14 \pi$.
We choose an integration path that has two segments passing through these saddle points at
an angle of~$\pm\tfrac14 \pi$ with respect to the real axis and with length~$\sqrt{2}L^{-\alpha}$,
where~$\alpha$ is yet to be chosen. The segments are joined by a vertical line, and
extended by vertical lines towards~$\pm\im\infty$.
Our path thus consists of the five segments (cf.~Figure~\ref{fig:2sp})
\begin{align*}
 \mathcal{C}_1:\quad s &= -L^{-\alpha} + \im t, && t\leq -L^{-1/2} - L^{-\alpha}, \\
 \mathcal{C}_2:\quad s &= -\im L^{-1/2} + (1+\im)t, && |t|\leq L^{-\alpha}, \\
 \mathcal{C}_3:\quad s &= L^{-\alpha} + \im t, && |t| \leq L^{-1/2}-L^{-\alpha}, \\
 \mathcal{C}_4:\quad s &= \im L^{-1/2} + (\im-1)t,&& |t|\leq L^{-\alpha}, \\ 
 \mathcal{C}_5:\quad s &= -L^{-\alpha} + \im t, && t\geq L^{-1/2} + L^{-\alpha}.
\end{align*}

\begin{figure}[t]
\begin{center}
 \includegraphics[width=7.3truecm, viewport = 10 0 300 250, clip]{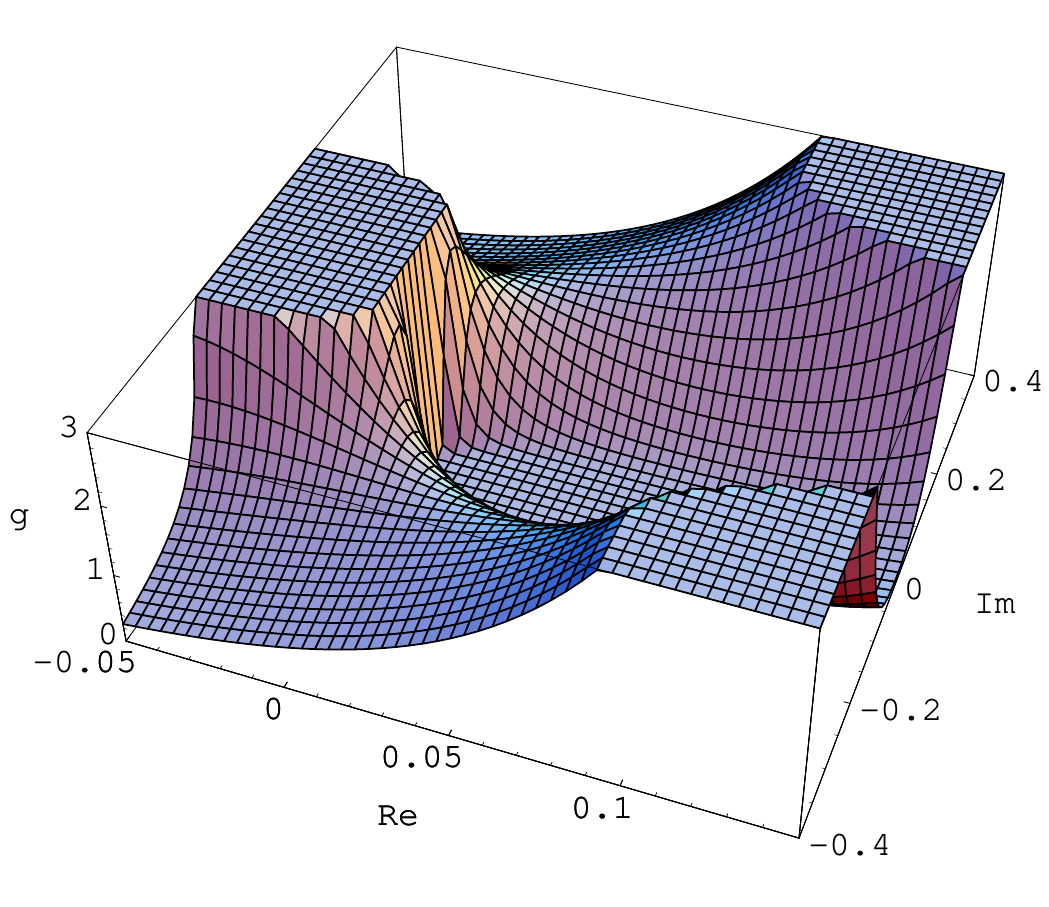}
 \includegraphics[width=4truecm]{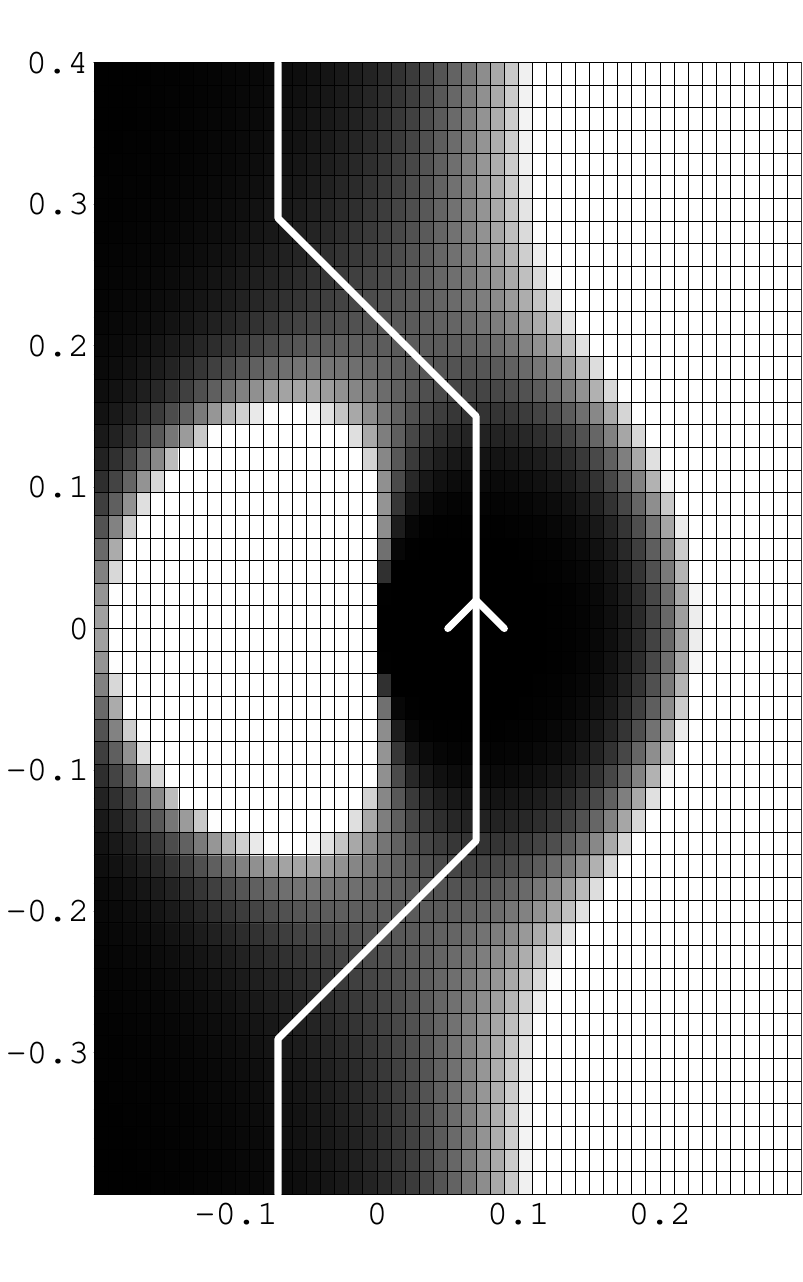}
 \caption{The landscape of $|z^s\e^{-1/s}/\sin\pi s|$, where $z=10^{10}$,
  and the new integration contour crossing the two approximate saddle points.}
 \label{fig:2sp}
\end{center}
\end{figure}

We will see that the exponent~$\alpha$ must satisfy the same bounds as in the previous
subsection,
i.e.~$\tfrac23<\alpha<\tfrac34$.
For the segment~$\mathcal{C}_2$, containing the lower saddle point, we again appeal to the
expansions from Figure~\ref{ta:expansions} and find
\[
 \frac{z^s \e^{-1/s}}{\sin \pi s} = \im\pi^{-1}L^{1/2} \exp(-2\im L^{1/2}-2L^{3/2}t^2)\cdot (1 + \Oh(L^{2-3\alpha})).
\]
Here we have set $s = -\im L^{-1/2} + (1+\im)t$, and assume that~$\alpha>\tfrac23$.
Since
\[
 \int_{-L^{-\alpha}}^{L^{-\alpha}} \exp(-2L^{3/2}t^2) \dd t \sim \tfrac{1}{\sqrt{2}}L^{-3/4}
  \int_{-\infty}^\infty \e^{-r^2} \dd r = \sqrt{\tfrac{\pi}{2}}L^{-3/4}
\]
for~$\alpha<\tfrac34$, we thus have
\[
 \int_{\mathcal{C}_2} \frac{z^s \e^{-1/s}}{\sin \pi s} \dd s = (\im-1) (2\pi)^{-1/2}
  L^{-1/4} \e^{-2\im L^{1/2}} \cdot (1+\Oh(L^{2-3\alpha})).
\]
The contribution of the upper saddle point,
\[
 \int_{\mathcal{C}_4} \frac{z^s \e^{-1/s}}{\sin \pi s} \dd s = (\im+1) (2\pi)^{-1/2}
  L^{-1/4} \e^{2\im L^{1/2}} \cdot (1+\Oh(L^{2-3\alpha})),
\]
is similarly found.
The dominant part of~\eqref{eq:lind2} is therefore
\begin{align}
 -\frac{1}{2\im} \int_{\mathcal{C}_2\cup\mathcal{C}_4} \frac{z^s \e^{-1/s}}{\sin \pi s} \dd s
 &= -\frac{1}{2\im} \left( \int_{\mathcal{C}_2} - \overline{\int_{\mathcal{C}_2}} \right) 
   = -\Im \left(\int_{\mathcal{C}_2} \frac{z^s \e^{-1/s}}{\sin \pi s} \dd s \right) \notag \\
 &=-(2\pi)^{-1/2} L^{-1/4}(\cos 2 L^{1/2} + \sin 2 L^{1/2}) \cdot (1+\Oh(L^{2-3\alpha})) \notag \\
 &= -\pi^{-1/2} L^{-1/4} \cos(2L^{1/2}-\tfrac14 \pi) \cdot (1+\Oh(L^{2-3\alpha})). \label{eq: main part}
\end{align}

It  remains to bound the integrals over~$\mathcal{C}_1$, $\mathcal{C}_3$, and $\mathcal{C}_5$.
The portion of~$\mathcal{C}_5$ with $t=\Im(s)\geq 1$ is $\Oh(\exp(-L^{1-\alpha}))$, by the same argument
as in the case of one saddle point.
Now consider the lower part of~$\mathcal{C}_5$, where we have $L^{-1/2}+L^{-\alpha}\leq t< 1$.
The factor $\pi/\sin \pi s$ is of order~$\Oh(L^\alpha)$.
Also, it is easy to see that~$|\exp(-1/s)|$ is a decreasing function of~$\Im(s)$ there. At the lower
endpoint of~$\mathcal{C}_5$, we estimate
\[
 \exp(-1/s) = \exp(L^{1-\alpha}-2L^{3/2-2\alpha}+\oh(1)).
\]
Since we have $z^s=\Oh(\exp(-L^{1-\alpha}))$ in~$\mathcal{C}_5$, this segment thus contributes only
$\exp(-2L^{3/2-2\alpha}+\oh(1))$ to the integral. Finally, 
we examine the segment~$\mathcal{C}_3$.
The factor~$|\exp(-1/s)|$ is an increasing function
of~$|\Im(s)|$ there. Hence it suffices to estimate
$\exp(-1/s)z^s$ at the upper endpoint of~$\mathcal{C}_3$, which is straightforward and shows
that the integral over~$\mathcal{C}_3$ is also negligible. This completes
the tail estimate.

Equation~\eqref{eq: main part} hence yields the result
\[
 \E(z;-1,-1) = -\tfrac{1}{\sqrt{\pi}} (\log z)^{-1/4} \cos\Bigl(2\sqrt{\log z} -\tfrac14 \pi\Bigr) +
  \Oh((\log z)^{-1/2+\varepsilon}),
  \qquad z\to\infty.
\]
The generalization to arbitrary negative parameters is as follows.

\begin{theorem}\label{thm:twosaddle}
 Let $c<0$ and $\theta<0$ be real numbers. Then
 \begin{multline*}
   \E(z;c,\theta)   = A_1 \exp\left(A_2 (\log z)^{\frac{\theta}{\theta-1}}\right) (\log z)^{\frac{\theta}{2(1-\theta)}} \cos\left(A_3
   (\log z)^{\frac{\theta}{\theta-1}} + A_4 \right)   \\
   + \Oh\left(\exp\left(A_2 (\log z)^{\frac{\theta}{\theta-1}}\right)  (\log z)^{\frac{\theta}{2(1-\theta)}-\mu} \right)
 \end{multline*}
 as $z\to \infty$ in an arbitrary sector with vertex at zero that does not contain the negative real axis.
 The constants are defined by
\[
\begin{array}{lllllll}
A_1 &=& \ds -(c\theta)^\frac{1}{2(\theta-1)} \sqrt{\tfrac{2}{\pi(1-\theta)}}, &&
A_2 &=& \ds (1-\theta^{-1})(c\theta)^{\frac{1}{1-\theta}} \cos \tfrac{\pi}{1-\theta}, \\
A_3 &=& \ds (1-\theta^{-1})(c\theta)^{\frac{1}{1-\theta}} \sin \tfrac{\pi}{1-\theta}, &&
A_4 &=& \ds \tfrac{\pi}{2(\theta-1)},
\end{array}
\]
 so that~$A_2$ is negative for $-1<\theta<0$, zero for $\theta=-1$, and positive for $\theta<-1$.
 The exponent~$\mu$ is as in Theorem~\ref{thm:saddle}.
\end{theorem}

\begin{proof}
The general proof is very similar to the special case $c=\theta=-1$ (see above),
upon taking into
account the following comments.
There might be more than two saddle points in general, but
we have to consider only the 
ones that form a conjugate pair having the largest real part, which are (approximately) at
$\exp(\pm \im \pi/(1-\theta))L^{1/(\theta-1)}$. The
saddle point  axes have the arguments $\pm\tfrac12 \pi(2-\theta)/(1-\theta)$.
When $-3\leq\theta\leq -1$, the proof proceeds as above. The parameter~$\alpha$, which governs the size
of the two contour segments containing the saddle points, must satisfy
\[
 \tfrac{3-\theta}{3(1-\theta)} < \alpha < \tfrac{2-\theta}{2(1-\theta)}.
\]

For $\theta<-3$, there is one minor problem in the tail estimate:
the decrease of
\begin{equation}\label{eq:osc tail}
 |\exp(c s^\theta)| = \exp(c |s|^\theta \cos(\theta \arg(s)))
\end{equation}
does not hold on the whole of~$\mathcal{C}_5$. In fact it is plausible that there are oscillations here
if~$|\theta|$ is large. What happens is that, moving upwards 
along~$\mathcal{C}_5$, the quantity $|\exp(c s^\theta)|$
decreases up to a local minimum, at~$s_0$ say. 
This results from elementary analysis;
the argument of~$s_0$ is $\arg(s_0)=2\pi/(1-\theta)$.
Now the integral over the portion of~$\mathcal{C}_5$ below~$s_0$ can be estimated as before, by taking
into account the length of this part of the contour
and the size of the integrand at the lower endpoint of~$\mathcal{C}_5$.
Above~$s_0$, the crude estimate $|\exp(c s^\theta)|\leq \exp(-c |s|^\theta)$
suffices.

Finally, when $-1<\theta<0$, the contour has to be slightly adjusted, as
the       two      saddle     points        at      $\exp(\pm      \im
\pi/(1-\theta))L^{1/(\theta-1)}$ have negative real  part, and we may
not push the contour over the singularity at zero.  Instead of joining
the two   central  segments~$\mathcal{C}_2$  and~$\mathcal{C}_4$  by a
vertical line,  we stretch each  of  them into the  right  half-plane,
stopping at $\Re(s)=L^{-\alpha}$. Then    we join them by a   vertical
line,  whose  upper   endpoint   we call~$s_1$.   The   value~$|\exp(c
s^\theta)|$  decreases as we  move   on~$\mathcal{C}_4$ to the  right,
until~$\Re(s)=0$, from   which point  it increases. Moreover,   as  we move
downwards      from~$s_1$, the quantity  $|\exp(c     s^\theta)|$        decreases
until~$\Im(s)=0$. Therefore, it suffices  to check that  the integrand
at~$s_1$ is  of growth  slower than  the central  part of the integral, a property
that is easily checked to hold.
\end{proof}

\subsection{Completion of the Proof of Theorem~\ref{cor-bell}}\label{se:thm compl}

It remains to consider the parameter region $\theta>1$, which is not covered by our
previous asymptotic
estimates of~$\E(z;c,\theta)$.
First, for parameter values $c<0$, $1<\theta$, which make~$\E(z;c,\theta)$ an entire function,
a formula similar to~\eqref{eq:saddle exp} holds. 
The exponential growth order is the same as in~\eqref{eq:saddle exp}, except for a different
constant in place of~$K_2$. 
In particular, the sequence~$\e^{cn^\theta}$ is not holonomic for these parameter values either, as 
$\theta/(\theta-1)\neq1$.
This asymptotic property is a special case of a result due to Valiron~\cite{Va14}, who used the Laplace method
to investigate the behaviour of~$\sum_{n\geq 0} \e^{-G(n)} x^n$ as $x\to\infty$,
where~$G$ is a smooth function that satisfies certain regularity conditions.
(Valiron's conditions actually require $1<\theta\leq2$, but his analysis is easily extended.)

Finally, in the remaining parameter range $c>0$, $\theta>1$, the sequence~$\e^{cn^\theta}$
grows faster than any power of~$n!$, which is incompatible with the growth of any
holonomic sequence. 
(In fact this observation shows that the (formal!) power series~$\E(z;c,\theta)$
does not even satisfy an algebraic differential equation~\cite{Ru89}.)

\section{\bf Asymptotics of Finite Differences}\label{dif-sec}
Beyond establishing non-holonomicity, the analysis  near infinity
of generating functions,
such as $\E(z;c,\theta)$, is also of interest in the estimation
of finite differences and related combinatorial sums.
Given a sequence~$(f_n)$, we shall refer to the derived sequence
\begin{equation}\label{dif0}
D_n[f]:=\sum_{k=0}^n \binom{n}{k}(-1)^k f_k
\end{equation}
as the sequences of \emph{differences}. (In the standard terminology,
we have $D_n[f]\equiv(-1)^n \Delta^n f_0$; see~\cite{Jordan65,Milne81}.)
The relation between~$f_n$ and $g_n:=D_n[f]$ is
translated at generating function level by the relation
\begin{equation}\label{euler}
g(z)=\frac{1}{1-z} f\left(-\frac{z}{1-z}\right) =
\frac{1}{1-z} \left(f_0 + F\left(\frac{z}{1-z}\right)\right),
\end{equation}
where $f(z),g(z)$ are the ``standard'' (i.e., non-alternating)
generating functions
\[
f(z):=\sum_{n\ge0} f_n z^n, \qquad
g(z):=\sum_{n\ge0} g_n z^n,
\]
and $F(z)\equiv f(-z)-f_0$ is the ``alternating'' 
generating function of~\eqref{gf0}.

Because of the alternation of signs in~\eqref{dif0} and the 
fact that the binomial coefficients become almost as large as~$2^n$, the
asymptotic estimation of differences is usually a non-trivial task.
Here, the ``surprise'' is the fact that, for many explicit and 
simple sequences~$f_n$, the corresponding~$g_n$ are much smaller
than~$2^n$: huge cancellations occur in~\eqref{dif0} (see, e.g., \cite{FlSe95},
for cases related to~$f_n=n^\alpha,\log n$, and so on).
For instance, with $f_n=\e^{1/n}$, $n\geq1$, we find
\[
g_1\doteq -2.71828, \quad
g_{10}\doteq -8.03246,  \quad
g_{100}\doteq -20.4159, \quad
g_{1000}\doteq -45.1379,
\]
and the sequence appears to grow rather slowly. For~$f_n=\e^{\sqrt{n}}$, it even 
appears  numerically
to tend slowly to~$0$.
(For recent estimates relative to zeta values and inverse zeta values, see
for instance,~\cite{FlVe08}.)
We now explain how a Lindel\"of type of analysis can serve to quantify
such phenomena.

The basic message of singularity 
analysis theory~\cite{FlOd90b,FlSe09,Odlyzko95}
is that the behaviour of a sequence is (usually) detectable 
from the singularities of its generating function.
Here, we should investigate the singularity of~$g(z)$ at $z=1$. 
Now, this singularity is tightly coupled with the one of $(1-z)g(z)$,
which by~\eqref{euler} depends on the behaviour of~$F(z)$ near~$+\infty$.
(This, by elementary properties of the conformal map $z\mapsto w=z/(1-z)$ and its
inverse $w\mapsto z=w/(1+w)$.)
Clearly, by considering specific analytic maps~$\sigma(z)$ 
(here: $\sigma(z)=z/(1-z)$), a large number of seemingly hard alternating sums
become asymptotically tractable.

\begin{corollary} \label{cor:diff} 
 The differences of the sequences
 $\e^{\pm \sqrt{n}}$ and~$\e^{\pm 1/n}$ have the following asymptotic behaviour.
 \begin{align} 
   \sum_{k=0}^n \binom{n}{k}(-1)^k \e^{\pm\sqrt{k}} &
  \sim -\frac{\pm 1}{\sqrt{\pi \log n}}, \label{eq:diff1} \\
   \sum_{k=1}^n \binom{n}{k}(-1)^k \e^{1/k} 
& \sim -\frac{\mathrm{e}^{2\sqrt{\log n}}}{2\sqrt{\pi}(\log n)^{1/4}}, \label{eq:diff2} \\
   \sum_{k=1}^n \binom{n}{k}(-1)^k \e^{-1/k} &
   = - \frac{\cos\left(2\sqrt{\log n}-\tfrac14 \pi\right)}{\sqrt{\pi}(\log n)^{1/4}}
    +\oh\left( (\log n)^{-1/4}\right). 
\label{eq:diff3}
 \end{align}
\end{corollary}
\begin{proof}
 Consider $f_n=\e^{\pm \sqrt{n}}$, and define~$g(z)$ as above.
 Then~\eqref{eq:hankel expans} implies
 \[
   g(z) \sim -\frac{\pm1}{\sqrt{\pi}} \frac{1}{1-z} \left( \log \frac{1}{1-z}\right)^{-1/2}, \qquad z\to1,
 \]
 whence~\eqref{eq:diff1} follows by the appropriate transfer theorem~\cite{FlOd90b,FlSe09,Odlyzko95}.
 In the cases $f_n=\e^{\pm 1/n}$ (with $f_0=0$), the growth
 of the difference generating functions, say~$g_1(z)$ and~$g_2(z)$, can be determined by Theorems~\ref{thm:saddle}
 and~\ref{thm:twosaddle}. 
 Note that both slowly varying and periodic functions
 can be subjected to singularity analysis~\cite{FlOd90b,FlSe09,Te07}, so that the formulas
 \begin{align*}
   g_2(z) &\sim -\frac{1}{2\sqrt{\pi}(1-z)}\exp\left(2\left(\log\frac{1}{1-z}\right)^{1/2}\right)
    \left(\log\frac{1}{1-z}\right)^{-1/4}, \\
   g_3(z) &= - \frac{1}{\sqrt{\pi}(1-z)} \left( \log \frac{1}{1-z} \right)^{-1/4}
   \cos \left(2\sqrt{\log\frac{1}{1-z}}-\frac{\pi}{4} \right) \\
    &\phantom{=} + \oh \left(\frac{1}{1-z}\left( \log \frac{1}{1-z} \right)^{-1/4} \right),
 \end{align*}
 yield~\eqref{eq:diff2} and~\eqref{eq:diff3}, respectively.
\end{proof}

The estimate~\eqref{eq:diff2} bears a striking formal 
resemblance with the growth of the average value of the multiplicative partition
function, which was found by Oppenheim~\cite{Op26} and Szekeres and Tur\'an~\cite{SzTu33}.
Indeed, only the sign and the exponent of~$\log n$ ($-\tfrac34$ instead of~$-\tfrac14$)
differ.

As an application of Corollary~\ref{cor:diff}, we note that Madsen~\cite{Ma93}
has considered generalized binomial distributions of the form
\[
 \mathbb{P}[X=x] = \binom{n}{x}\sum_{j=0}^{n-x}\binom{n-x}{j}(-1)^j \pi_{x+j}, \qquad x\in\{0,\dots,n\},
\]
where, for instance, he sets $\pi_k=\exp((\log p)k^a)$ with $0\leq a\leq 1$ and $0<p<1$.
We can then describe, by an obvious extension of~\eqref{eq:diff1}, the way the 
probability mass function behaves for large parameters~$n$:
\[
 \mathbb{P}[X=x] \mathop{\sim}_{n\to\infty}
 \begin{cases}
\ds   -\frac{\log p}{\Gamma(1-a)(\log n)^a} & x = 0\\
\ds   -\frac{a\log p}{x\Gamma(1-a)(\log n)^{a+1}} & x \geq 1
 \end{cases}.
\]
In particular, there is no limit distribution, as $n\to\infty$.

\section{\bf Behavior of~$\E(z;c,\theta)$ at its Dominating Singularity}\label{se:minus one}

Although not related to non-holonomicity or the Ford-Lindel\"of technique, it seems natural
to complement the asymptotic results we have obtained for the function~\eqref{defe} by
investigating its dominating singularity, located at~$z=-1$.
The asymptotic behaviour there is comparatively easy to determine. To begin with,
for $\theta<0$ and any real~$c$ we can rewrite
\begin{equation}\label{eq:F Li}
 \E(z;c,\theta) = \sum_{n\geq 1}\sum_{k\geq 0} \frac{c^k}{k!} n^{k\theta} (-z)^n = \sum_{k\geq 0}\frac{c^k}{k!}
  \Li_{-k\theta}(-z), \qquad |z|<1,
\end{equation}
as a sum of polylogarithms
\[
 \Li_\alpha(z) = \sum_{n\geq1} \frac{z^n}{n^\alpha},
\]
whose asymptotic behaviour at $z=1$ is known. The shape of the
asymptotic expansion of $\Li_\alpha$ depends on whether~$\alpha$ is an integer~\cite{Flajolet99,FlSe09}.

\begin{proposition}
 Let $c$ be a real number and $\theta$ be a negative real number. Then the asymptotic expansion
 of~$\E(z;c,\theta)$ at $z=-1$ is 
obtained by transporting the expansions of~$\Li_{-k\theta}$
 into~\eqref{eq:F Li}.
\end{proposition}

Adding infinitely many asymptotic expansions termwise can be easily justified here by truncating the
expansions and appealing to uniform convergence, which permits us
to exchange limit and summation.
For instance, if~$\alpha\geq1$ is an integer, we have
\[
 \Li_\alpha(z) = \frac{(-1)^\alpha}{(\alpha-1)!} w^{\alpha-1}(\log w - H_{\alpha-1})
 + \sum_{j\geq0, j\neq \alpha-1} \frac{(-1)^j}{j!} \zeta(\alpha-j)w^j,
\]
where $w = -\log z$ and $H_{\alpha-1}$ is a harmonic number.
For the parameter values $c=1$, $\theta=-1$ we thus obtain
\begin{align*}
 \E(z;c=1,\theta=-1)&=\sum_{n\geq 1} \e^{1/n}(-z)^n = \sum_{k\geq 0}\frac{1}{k!}\Li_k(-z) \\
 &= \frac{1}{1+z} + \log\frac{1}{1+z} + C + \Oh((1+z)\log\frac{1}{1+z}),
\end{align*}
where 
$C= -1+\sum_{k\geq 2}\zeta(k)/k!\approx 0.078189$.


We proceed to the case $c>0$ and $0<\theta<1$. 
For fixed~$z$ inside the unit disk, the summands
have a peak at some~$n$ and then decrease rapidly, which makes the Laplace method
a natural approach for estimating the sum. Our illustrative example is $c=1$ and $\theta=\tfrac12$.
The summands~$\e^{\sqrt{n}} (-z)^n$ have their peak near~$u:=\lfloor \tfrac14 v^{-2} \rfloor$,
where~$v:=-\log|z|$. This leads to the lower bound $\exp(\tfrac14 v^{-1})(1+\Oh(v))$, which was
already noted by Borel~\cite[p.~69]{Borel02}. It is straightforward to determine
the second order approximation of the summand near~$n=u$, and to estimate the tails
of the original sum and the second order approximation.
What we find is
\begin{align*}
 \E(z;c=1,& \theta=\tfrac12) =
 \sum_{n\geq 1} \e^{\sqrt{n}} (-z)^n \\
 &= \frac{\sqrt{\pi}\e^{-1/8}}{(1+z)^{3/2}}
 \exp\Bigl( \frac{1}{4(1+z)} \Bigr) \left( 1+ \Oh((1+z)^{1/2-\varepsilon})\right), \qquad z\to -1^+\ \text{in}\ \mathbb{R}.
\end{align*}
The (again straightforward) generalization from $\sqrt{n}$ to $cn^\theta$ reads as follows.

\begin{proposition}
 Let $c$ be a positive real number, $0<\theta<1$, and $\varepsilon>0$. Then
 \begin{equation}\label{eq:lap gen}
   \E(z;c,\theta) = C_1 (1+z)^{\frac{2-\theta}{2(\theta-1)}} \exp(C_2 v^{\frac{\theta}{\theta-1}})
    \left(1 + \Oh((1+z)^\mu) \right)
 \end{equation}
 as~$z$ tends to~$-1^+$ in~$\mathbb{R}$,
 where $v:=-\log |z|$. The constants $C_1$, $C_2$, $C_3$ are positive and
given by
\[
\begin{array}{llllllll}
   C_1 &:=&\ds  \sqrt{2\pi}(1-\theta)^{-1/2}(c\theta)^{\frac{1}{2(1-\theta)}}, &&&
   C_2 &:=&\ds  \tfrac{1-\theta}{\theta} (c\theta)^{\frac{1}{1-\theta}}, \\
   C_3 &:=&\ds  (c\theta)^{\frac{1}{1-\theta}},
\end{array}
\]
 and the exponent of the relative error estimate is
 \[
   \mu := \min\{ \tfrac{\theta}{2(1-\theta)} - \varepsilon,  1\}.
 \]
\end{proposition}


Finally, we consider $c<0$ and $0<\theta<1$. Then the series
\[
 \E(z;c,\theta)=\sum_{n\geq 1} \e^{cn^\theta} (-z)^n
\]
converges at $z=-1$. Inside the unit circle we may differentiate it termwise
arbitrarily many times, and the series thus obtained converge at $z=-1$, too.
By Abel's convergence theorem, these values equal the limits of the derivatives as $z\to -1^+$.
The (divergent) formal Taylor series of~$\E(z)$ at $z=-1$ obtained in this way is
an asymptotic series for the function~\cite[p.~30]{Ford60}, which yields the following result.
\begin{proposition}
Suppose that $c<0$ and $0<\theta<1$. Then
\begin{equation}\label{eq:F series}
 \E(z;c,\theta) \mathop{\sim}_{z\to -1^+}   u_0 + u_1(1+z) + u_2(1+z)^2 + \dots
\end{equation}
as~$z$ tends to~$-1^+$ in~$\mathbb{R}$, where the coefficients are given by
\begin{equation}\label{eq:u_k}
 u_k := \frac{1}{k!}\lim_{z\to -1^+} \frac{\dd^k}{\dd z^k} \E(z)
  =  (-1)^k \sum_{n\geq 1}\binom{n}{k} \exp(cn^\theta).
\end{equation}
\end{proposition}
Note that the series~\eqref{eq:F series} does not converge in any neighborhood of~$z=-1$, since Stirling's formula yields
\[
 |u_k| \geq \binom{n(k)}{k} \e^{cn(k)^\theta} \gg k^{(1/\theta-1-\varepsilon)k},
\]
where $n(k)=\lfloor (-k/c\theta)^{1/\theta}\rfloor$ approximates the index of the largest
summand in~\eqref{eq:u_k}. Hence $z=-1$ is indeed a singularity.

\section{\bf Conclusion}

We have revisited a classical method for the analytic continuation of power series
beyond their disc of convergence,
with the goal of obtaining asymptotic expansions
and comparing them against the
possible expansions of holonomic functions.
Our estimates can be used as building blocks for the asymptotic analysis of
more complicated functions than those we have explicitly mentioned. For instance,
the expansion of functions such as
\[
 \sum_{n\geq0} \frac{\sqrt{n}\ \e^{\sqrt{n}}(-z)^n}{2^n+n^2}
\]
at~$+\infty$ readily springs from Lemma~\ref{le:alg} in~\S\ref{alg-sec},
and series in the spirit of
\[
 \sum_{n\geq0} \e^{1/n + 1/\sqrt{n}} (-z)^n
\]
can be analysed similarly as the ones in~\S\ref{se:saddle} and~\S\ref{se:two saddle}.
As regards proving non-holonomicity, our results compete with those of Bell \emph{et al.}~\cite{BeGeKlLu08},
who also deal with sequences having an analytic lifting.
Roughly speaking, our approach is more versatile for meromorphic functions,
equivalent in the algebraic case, and less flexible in the
presence of essential singularities.

Note, however, that neither we nor 
Bell \emph{et al.}~\cite{BeGeKlLu08} can show non-holonomicity
of sequences whose extrapolating
function is entire. For instance, we leave the non-holonomicity of 
sequences like $\cos(\sqrt{n})$ and $\cosh(\sqrt{n})$
as an open problem, since their analytic liftings, namely, $\cos(\sqrt{s})$ and $\cosh(\sqrt{s})$,
have no singularity at a finite distance.

\def\cprime{$'$}

\end{document}